\documentclass[a4paper,10pt]{scrartcl}

\RequirePackage{enumerate}
\RequirePackage{verbatim}
\RequirePackage{hyperref}

\RequirePackage{amsmath,amsfonts,amssymb}
\RequirePackage[mathcal,mathbf]{euler}

\usepackage{latexsym,amscd,amsthm}
\usepackage[all]{xy}

\usepackage{amstext,amsopn}
\usepackage{graphicx}
\usepackage[dvips]{color}
\usepackage{epsfig}
\usepackage{psfrag}
\usepackage[all]{xy}
\usepackage{amscd}
\usepackage{amssymb}
\usepackage{fullpage}
\usepackage{url}
\usepackage{manfnt} 
\usepackage{xcolor}
\usepackage{supertabular}

\usepackage{tikz}

\def\dar[#1]{\ar@<2pt>[#1]\ar@<-2pt>[#1]}
\entrymodifiers={!!<0pt,0.7ex>+}

\newtheorem{Thm}{Theorem}[section]
\newtheorem{Prop}[Thm]{Proposition}
\newtheorem{Lem}[Thm]{Lemma}
\newtheorem{Cor}[Thm]{Corollary}

\theoremstyle{remark}
\newtheorem{Rem}[Thm]{Remark}

\theoremstyle{definition}

\newtheorem{Def}[Thm]{Definition}
\newtheorem{Exa}[Thm]{Example}

\newcommand{\C}{\mathbb{C}}

\newcommand{\R}{\mathbb{R}}

\newcommand{\Z}{\mathbb{Z}}

\DeclareMathOperator{\ad}{ad}
\DeclareMathOperator{\Diff}{Diff}
\DeclareMathOperator{\Aut}{Aut}

\DeclareMathOperator{\Fix}{Fix}
\DeclareMathOperator{\Hom}{Hom}

\newcommand{\llb}{[\![}
\newcommand{\rrb}{]\!]}
\newcommand{\g}{\mathfrak{g}}
\newcommand{\h}{\mathfrak{h}}

\newcommand{\cO}{\mathcal{O}_\lambda}

\author{Damien Calaque and Florian N\"af}
\title{A trace formula for the quantization of coadjoint orbits}
\date{}
\begin{document}

\maketitle

\begin{abstract}
\noindent{\bf Abstract.} The main goal of this paper is to compute the characteristic class of the Alekseev-Lachowska $\star$-product on coadjoint orbits. 
We deduce an analogue of the Weyl dimension formula in the context of deformation quantization. 
\end{abstract}

\setcounter{tocdepth}{2}
\tableofcontents

\section*{Introduction}
\addcontentsline{toc}{section}{Introduction}

The Algebraic Index Theorem of Fedosov \cite{FeBook} and Nest-Tsygan \cite{NT} provides in principle a powerful tool to get various kinds of trace formul\ae~for non-commutative 
algebras arising as deformation quantization of symplectic manifolds. However, it involves a characteristic class associated with any such deformation quantization, which might be very 
hard to compute explicitly. 

\medskip

The purpose of this note is to provide a computation of the characteristic class associated with a deformation quantization of a semi-simple coadjoint orbit $\mathcal O_\lambda\subset\mathfrak g^*$, 
that has been constructed by Alekseev and Lachowska \cite{AL} by means of representation theory. More precisely, we show that the characteristic class of the Alekseev-Lachowska quantization is 
$$
-\lambda+ t \, i \, \rho \in \Big(\frac{\mathfrak{h}}{[\h,\h]}\Big)^*\llb t \rrb \rightarrow H^2\Big(\mathcal O_\lambda,\R\llb t \rrb\Big)\,,
$$
where $\rho$ is the half-sum of positive roots. 
 
\medskip

We hope that the methods presented in this short note can be used and applied to other more appealing open problems, such as 
computing the characteristic class of some combinatorially defined quantization of the moduli space of flat connections 
on a surface (see e.g.~\cite{AGS}) and thus leading to a proof of Conjecture 7.2 in \cite{RS}. 

\subsubsection*{Organization of the paper}

The paper begins with a short recollection on deformation quantization: we recall the standard existence and classification results for $\star$-products, 
as well as the trace formula of Fedosov and Nest-Tsygan. We continue the recollection in Section 2 about deformation quantization in the presence 
of a polarization or a separation of variables, after Bressler-Donin and Karabegov. Section 3, which is the core of this note, is devoted to the definition and properties 
of the Alekseev-Lachowsla $\star$-product: we prove that it has a separation of variables and a strong quantum momentum map, and we use these properties to compute 
its characteristic class. In the last Section we provide a purely algebraic expression for $Tr(1)$, which can be understood as an analog of the Weyl dimension formula 
in the context of deformation quantization. 

\subsubsection*{Notation}

If $E$ is a vector bundle over a manifold $M$ then $\mathcal E$ denotes its sheaf of sections. 

\subsubsection*{Acknowledgements} 

We thank Anton Alekseev for suggesting the problem to us. 
This work was done when the second author was an undergraduate student at ETH Z\"urich, as part of his Master Thesis.  
The work of D.C.~has been partially supported by a grant from the Swiss National Science Foundation 
(project number $200021\underline{~}137778$). 


\section{Short recollection on deformation quantization}

Recall that a {\it deformation quantization} of a symplectic manifold $(M,\omega)$ is an associative $\mathbb{C} \llb t \rrb$-linear product $\star$ 
on $C^\infty(M,\C) \llb t \rrb$ of the form
$$
f\star g =fg+tB_1(f,g)+t^2B_2(f,g)+\cdots\qquad\big(f,g\in C^\infty(M,\C)\big)\,,
$$
where $B_i$'s are bidifferential operators on $M$ vanishing on constants such that
$$
\{f,g\}=\frac{1}{t}(f\star g-g\star f)_{|t=0}\qquad\big(f,g\in C^\infty(M,\C)\big)\,,
$$
where the Poisson bracket $\{\cdot,\cdot\}$ is the one induced by the symplectic structure $\omega$. 
We call $\star$ a {\it $\star$-product} and it turns the sheaf $C^\infty_M\llb t \rrb$ of 
series of $\C$-valued smooth functions into a sheaf of associative algebras, denoted $\mathcal A_t=\big(C^\infty_M\llb t \rrb,\star\big)$. 

\begin{Exa}[Moyal-Weyl quantization]
Let $M$ be an open subset of a symplectic vector space. The series $\exp(\frac{t}{2}\pi)$, where $\pi=\omega^{-1}$ is the Poisson bivector associated with $\omega$, 
defines a $\star$-product. 
\end{Exa}
Two deformation quantizations $\star_1$ and $\star_2$ are said to be equivalent (or isomorphic) if there is a series 
$D=id+tD_1+t^2D_2+\cdots$ of differential operators such that $D(f\star_1g)=D(f)\star_2D(g)$. 
We write $\underline{\Aut}(\mathcal A_t)$ for the sheaf of self-equivalences of $\star$. 

\begin{Thm}[\cite{Fe,DWL}]
Any symplectic manifold admits a deformation quantization. 
\end{Thm}
Note that the above references \cite{Fe,DWL} exhibit a distinguished equivalence class of $\star$-products (the same in both constructions, see \cite{Deligne}). 
We call it the {\it canonical class}. 

\subsection{Classification}\label{ssec-classif}

Any two $\star$-products are locally equivalent. Hence any $\star$-product can be described by means of a fixed one 
$\star$ and local equivalences. Its equivalence class is then uniquely defined by the corresponding 
first cohomology class with values in the sheaf $\underline{\Aut}(\mathcal{A}_t)$. 

In more abstract terms, local properties of $\star$-products show that we have a gerbe of $\star$-products over our manifold. 
By the above existence Theorem we get a global section $\star$ of this gerbe. In this case, equivalence classes of global sections 
are in one-to-one correspondence with
$$
H^1(M,\underline{\Aut}(\mathcal{A}_t))\,.
$$

Note that the automorphism group is pro-unipotent. In particular any automorphism can be written as $\exp(\phi)$ for a nilpotent derivation of $\mathcal{A}_t$. 
Moreover, it can easily be seen that locally, any such derivation is inner. Therefore, we get the following short exact sequence of sheaves of groups 
(and central extension)
$$
1 \longrightarrow 1+t\underline{\C}\llb t \rrb \longrightarrow 1+t\mathcal{A}_t \longrightarrow  \underline{\Aut}(\mathcal{A}_t) \longrightarrow 1
$$
This induces a boundary morphism in non-abelian cohomology
$$
H^1(M, \underline{\Aut}(\mathcal{A}_t)) \rightarrow H^2(M,1+t\underline{\C}\llb t \rrb)\cong tH^2(M,\C)\llb t \rrb\,,
$$
which happens to be an isomorphism because $\mathcal{A}_t$ is soft (as sheaf of $\C\llb t \rrb$-modules) and therefore is acyclic. 
Hence the set of equivalence classes of $\star$-products is an affine space modeled on $tH^2(M,\C)\llb t \rrb$. 

We declare by convention that the characteristic class of the canonical class is $[\omega]\in H^2(M,\mathbb{C})$. Once this convention set up, 
the set of equivalence classes is canonically identified with $[\omega]+tH^2(M,\C)\llb t \rrb$. 

The {\it characteristic class} of a $\star$-product is the image in $[\omega]+tH^2(M,\C)\llb t \rrb$ of its equivalence class. 

\subsection{Canonical trace}

On the Moyal-Weyl quantization of an open subset of a symplectic vector space, one can define the following linear functional on compactly supported functions: 
\begin{equation}\label{localtrace}
\centering
Tr(f) = \frac{1}{(2 \pi t)^n} \int_{\mathbb{R}^{2n}} f \frac{\omega^n}{n!}, \quad \text{for } f \in C^\infty_c(\R^{2n}) \llb t \rrb.
\end{equation}
It vanishes on $\star$-commutators; it is actually, up to normalization, the unique linear functional vanishing on $\star$-commutators. 
We hence call $Tr$ the {\it canonical trace}. 
Since all $\star$-products are locally equivalent to a Moyal-Weyl product, one can extend this definition to global functions with compact support 
as long as it is invariant under auto-equivalences of the Moyal-Weyl product. This is indeed the case, as any derivation (and thus any auto-equivalence) 
of the Moyal-Weyl $\star$-product is inner. 
\begin{Def}
The unique $\C(\!(t)\!)$-linear functional $Tr: C^\infty_c\big(M,\C(\!(t)\!)\big) \rightarrow \C (\!(t)\!)$ that vanishes on $\star$-commutators and is locally given 
by \eqref{localtrace} is called the {\it canonical trace}.
\end{Def}
With this definition we can now cite the Fedosov-Nest-Tsygan theorem (see \cite{FeBook,NT}). 
\begin{Thm}
Let $\mathcal{A}_t=\big(C^\infty(M)\llb t \rrb,\star)$ be a deformation quantization of a compact symplectic manifold $(M,\omega)$ with characteristic class $\theta$, then
\begin{equation}\label{FNT}
Tr(1) = \frac{1}{(2\pi)^n} \int_M e^{\frac{\theta}{t}} \ \hat{A}(TM,\omega),
\end{equation}
where $\hat{A}$ is the multiplicative genus induced by the power series $\left(\frac{z/2}{\sinh(z/2)}\right)^{\frac{1}{2}}$, that is
$$
\hat{A}(TM,\omega) = \text{$\det$}^\frac{1}{2}\frac{R/2}{\sinh(R/2)}
$$
for the curvature $R$ of a symplectic connection on $TM$.
\end{Thm}


\section{Polarized deformation quantization after Bressler-Donin}

Recall that a {\it polarization} on a symplectic manifold $(M,\omega_0)$ is an integrable Lagrangian subbundle $P_0$ of the complexified tangent bundle $T^\C M$. 
Let $\mathcal{O}_0$ denote the sheaf of functions that are locally constant along $P_0$:
$$
\mathcal O_0=\left\{f\in C^\infty_M\,\big|\,df_{|P_0}=0\right\}\,.
$$
It is a subsheaf of $C^\infty_M$ satisfying $\{\mathcal{O}_0,\mathcal{O}_0\}=0$. 
\begin{Exa}[Cotangent bundle]
Consider $M=T^*X$ together with its canonical symplectic structure $\omega_0=dp_i\wedge dq_i$. 
Observe that the bundle $L=T^\pi M$ of vertical vectors, where $\pi:T^*X\to X$ is the obvious projection, defines a polarization. 
Then $\mathcal O_0=\pi^{-1}C^\infty_X$ is the sheaf of functions that are constant in the fibers (i.e.~functions depending on local variables $q_i$'s only). 
\end{Exa}

The above is actually an example of a {\it real} polarization (we didn't need to complexify the tangent bundle). 
Below we give an example of a polarization which is not a real one. 

\begin{Exa}[Coadjoint orbit]\label{exa-ca}
Let $\mathfrak g$ be a Lie algebra and $\mathcal{O}_\lambda$ be the coadjoint orbit of $\lambda\in\mathfrak g^*$.
We have that $\mathcal{O}_\lambda = G/H$, $H = \{ g \in G \big| g.\lambda = \lambda \}$, is endowed with a $G$-invariant symplectic structure given by
$\omega:  \g / \mathfrak{h} \times \g / \mathfrak{h} \rightarrow \R \,;\,(x,y) \mapsto \lambda([x,y])$. 
Here we have identified the tangent bundle of $\mathcal{O}_\lambda$, which is $G$-equivariant with fiber $\g/\mathfrak{h}$, with $G \times_H \g/\mathfrak{h}$. 
It follows from \cite[(7.3)]{Nomizu} that for any splitting $\iota : \g/\mathfrak{h} \hookrightarrow \g$ there exist a basis of local vector fields 
$(x_i)_i$ whose Lie bracket is given by
$$
[x_i,x_j](\lambda) = \left[\iota\big(x_i(\lambda)\big),\iota\big(x_j(\lambda)\big)\right] + \mathfrak{h} \in \g/\h.
$$
In particular, any $\mathfrak{n} \subset \g^\C$ such that $\mathfrak{n}+\h^\C$ is a Lie subalgebra of $\g^\C$ induces an involutive subbundle. 
If $\mathfrak{n}+\h^\C$ furthermore is a Lagrangian subspace of $\big((\g/\h)^\C, \lambda([\cdot,\cdot])\big)$, then the resulting involutive subbundle defines a polarization. 
\end{Exa}
An $\mathfrak{n}$ as in the above example can be shown to exist for semi-simple coadjoint orbits. 
\begin{Exa}[Semi-simple coadjoint orbit]\label{exa-ssca}
Let us assume that $\g$ is semi-simple and that $\lambda$ corresponds to a semi-simple element $x_\lambda \in \g$ via the Killing form. 
Thus, we can identify $\h$ with the centralizer of $x_\lambda$, that is the kernel of $\ad(x_\lambda)$. 
Since $\ad(x_\lambda)$ is $H$-invariant and semi-simple, its image constitutes an $H$-invariant complement to $\h$, showing that $\cO$ is reductive. 
Since $x_\lambda$ is semi-simple it is contained in a Cartan subalgebra. We then get the root space decomposition of $\g^\C$ inducing one on $(\g/\h)^\C$. 
Let $\Delta$ denote the set of roots $\alpha$ such that $\alpha(x_\lambda) \neq 0$:
$$
(\g/\h)^\C = \bigoplus_{\alpha \in \Delta} \C e_{\alpha}\,.
$$
To construct a polarization, we notice that the eigenvalues of $\ad(x_\lambda)$ occur in pairs of opposite signs. 
We hence have a decomposition $\Delta = \Delta_- \cup \Delta_+$ such that 
$$
\{\alpha(x_\lambda) |\, \alpha \in \Delta_-\} \cap \{\alpha(x_\lambda) |\, \alpha \in \Delta_+\} = \emptyset\,.
$$
In particular, $\mathfrak{n}_- := \bigoplus_{\alpha \in \Delta_-} \C e_{\alpha}$ is $H$-invariant and induces a polarization.
\end{Exa}

\begin{Def}
A deformation quantization of a polarized symplectic manifold $(M,\omega_0,P_0)$ is a pair $(\mathcal{A}_t,\mathcal{O}_t)$ such that: \\[-0.5cm]
\begin{itemize}
	\item $\mathcal{A}_t$ is a deformation quantization of $(M,\omega_0)$: $\mathcal{A}_t=\big(C^\infty_M \llb t \rrb,\star\big)$. \\[-0.6cm]
	\item $\mathcal{O}_t$ is a commutative subsheaf of $\mathcal{A}_t$ that consists of functions vanishing along a deformation $P_t$ of $P_0$.	
	More precisely, $\mathcal{P}_t$ is a locally topologically free $C^\infty_M \llb t \rrb$-module, such that $\mathcal P_t/(t)\cong \mathcal P_0$ and 
	$\mathcal{O}_t=\left\{f \in C^\infty_M \llb t \rrb \,\big|\, df_{|P_t} = 0\right\}$. \\[-0.5cm]
\end{itemize}
\end{Def}

\subsection{Characteristic class of a polarized $\star$-product}

In this Subsection we quote a result from \cite{BD} which helps to compute the characteristic class of a polarized $\star$-product. 
Let $(M,\omega_0,P_0)$ be a polarized symplectic manifold, where we further assume one of the following conditions: \\[-0.5cm]
\begin{itemize}
\item $M$, $\omega_0$ and $P$ are analytic. \\[-0.6cm]
\item $P\cap\bar{P}$ has constant rank and $P+\bar{P}$ is involutive.
\end{itemize}
Let $(\mathcal{A}_t,\star)$ be a polarized quantization of $(M,\omega_0,P_0)$, and define: \\[-0.5cm]
\begin{itemize}
\item $\mathcal{F}(\mathcal{O}_t):=\big\{ f \in \mathcal{A}_t\,\big|\,\frac{1}{t}[f, \mathcal{O}_t]_\star \subset \mathcal{O}_t \big\}$, \\[-0.6cm]
\item $\mathcal{T}_{\mathcal{O}_t}:=\underline{\Hom}_{\mathcal{O}_t}(\Omega^1_{\mathcal{O}_t/\C}, \mathcal{O}_t)=\underline{{\rm Der}}_\C(\mathcal O_t,\mathcal O_t)$.
\end{itemize}
\begin{Prop}[\cite{BD},Proposition 3.7]\label{prop3.16}
The sequence
$$
0 \longrightarrow \mathcal{O}_t \longrightarrow \mathcal{F}(\mathcal{O}_t) \longrightarrow \mathcal{T}_{\mathcal{O}_t} \longrightarrow 0
$$
is a locally split exact sequence of locally free $\mathcal{O}_t$-modules and Lie algebras. 
Moreover there exists, in a neighbourhood of every $m\in M$, local sections $a_1,\cdots,a_n$ of $\mathcal{O}_t$ independent of $t$ and pairwise commuting 
local sections $f_1,\cdots,f_n$ of $\mathcal{A}_t$ such that: \\[-0.5cm]
\begin{itemize}
\item $da_i$ form a local basis of $P_t^\perp$. \\[-0.6cm]
\item $\frac{1}{t}[a_i,f_j]_\star= \delta_{ij}$.\\[-0.6cm]
\item $1,f_1,\cdots,f_n$ form a local basis of the $\mathcal{O}_t$-module $\mathcal{F}(\mathcal{O}_t)$. \\[-0.6cm]
\item The morphism $\mathcal{F}(\mathcal{O}_t)_{|U} \rightarrow \bigoplus_{i=1}^n (\mathcal{O}_t)_{|U}$ ($U$ being an open neighbourhood of $m$) defined by 
\begin{equation}\label{eq-map}
g\longmapsto\left(\frac{1}{t}[g,a_1]_\star,\cdots, \frac{1}{t}[g,a_n]_\star\right)
\end{equation}
descends to a local isomorphism of $\mathcal{O}_t$-modules $(\mathcal{T}_{\mathcal{O}_t})_{|U}\rightarrow \bigoplus_{i=1}^n (\mathcal{O}_t)_{|U}$.
\end{itemize}
\end{Prop}
Such a local splitting then forms a \v{C}ech-$1$-cocycle with values in $\Omega_{\mathcal{O}_t}^{1,\text{cl}}$ and via the boundary homomorphism 
in cohomology of the short exact sequence
$$
0 \longrightarrow \underline{\C}\llb t \rrb \longrightarrow \mathcal{O}_t \longrightarrow \Omega_{\mathcal{O}_t}^{1,\text{cl}} \longrightarrow 0
$$
it defines a characteristic class $cl(\mathcal{A}_t,\mathcal{O}_t) \in H^2(M,\C\llb t \rrb)$. 
\begin{Thm}[\cite{BD},Theorem 4.6]\label{BDthm}
The characteristic class of a polarized deformation is 
$$
cl(\mathcal{A}_t,\mathcal{O}_t) - \frac{t}{2}c_1(P_0)\,.
$$
\end{Thm}
This result had been previsouly proved in \cite{Ka} for $\star$-products with separation of variables, which we define in the next Subsection. 

\subsection{Separation of variables}

Let $(M,\omega_0)$ be a symplectic manifold together with two transverse polarizations $P_0,Q_0\subset T^\C M$: more precisely 
we assume that the map $\mathcal{O}_0\otimes \widetilde{\mathcal{O}}_0\rightarrow C^\infty_M$ has dense image on $\infty$-jets at every point, where 
$\mathcal O_0=\big\{f \in C^\infty_M \, \big| \, df_{|P_0} = 0 \big\}$ and $\widetilde{\mathcal{O}}_0=\big\{f \in C^\infty_M \, \big| \, df_{|Q_0} = 0 \big\}$. 
We say that $M$ has a {\it separation of variables}. 
\begin{Exa}[Semi-simple coadjoint orbits]
We borrow the notation from Example \ref{exa-ssca}. The $\h^\C$-module $\mathfrak{n}_+=\sum_{\alpha\in\Delta_+}e_\alpha$ 
induces another polarization which is transverse to the one induced by $\mathfrak{n}_-$. Together they provide a separation of variables on $\cO$. 
\end{Exa}
\begin{Def}
A deformation quantization of a symplectic manifold with separation of variables $(M,\omega_0,P_0,Q_0)$ is a triple $(\mathcal A_t,\mathcal O_t,\widetilde{\mathcal O}_t)$ 
such that: \\[-0.5cm]
\begin{itemize}
	\item $(\mathcal A_t,\mathcal O_t)$ is a deformation quantization of $(M,\omega_0,P_0)$. \\[-0.6cm]
	\item $(\mathcal A_t,\widetilde{\mathcal O}_t)$ is a deformation quantization of $(M,\omega_0,Q_0)$. \\[-0.6cm]
	\item $f \star g = fg$ for $f \in \mathcal{O}_t$ and $g \in C^\infty_M$. \\[-0.6cm]
	\item $f \star g = fg$ for $f \in C^\infty_M$ and $g \in \widetilde{\mathcal{O}}_t$. 
\end{itemize}
\end{Def}


\section{Alekseev-Lachowska $\star$-product}

Let $\g$ be a semi-simple Lie algebra and $\lambda \in \g^*$ a semi-simple element. 
Following \cite{AL} we construct a $\star$-product on $\mathcal{O}_\lambda$.
Let $\h = \{ x \in \g \big| \, x.\lambda = 0\}$ be the subalgebra of $\g$ fixing $\lambda$. 
If $\lambda$ is regular (i.e.~$\dim \Fix_\lambda$ is minimal among all $\lambda \in \g^*$) then $\h^C$ is a Cartan subalgebra of $\g^\C$, 
else there exists a Cartan subalgebra in $\h^\C$. Borrowing the notation from Example \ref{exa-ssca} we get a root space decomposition
$$
\g^\C = \mathfrak{n}_- \oplus \h^\C \oplus \mathfrak{n}_+ 
= \bigoplus_{\alpha \in \Delta_+} \C e_{-\alpha} \oplus \h^\C \oplus \bigoplus_{\alpha \in \Delta_+} \C e_{\alpha}\,.
$$
Furthermore, we can normalize the root vector $e_\alpha$ such that
$$
\omega(e_{-\alpha} , e_{\beta}) = \lambda\big([e_{-\alpha}, e_{\beta}]\big) = \delta_{\alpha, \beta}\,.
$$
Set $\mathfrak{p}_\pm = \h^\C \oplus \mathfrak{n}_\pm$ and define the generalized Verma modules
$$
M_\lambda^\pm = U(\g^\C) \otimes_{U(\mathfrak{p}_\pm)} \C_{\pm\lambda}\,.
$$
Observe that $M_\lambda^+ \otimes_{U(\g^\C)} M_\lambda^-$ can be canonically identified with $\C$, therefore there exists a unique $\g^\C$-invariant pairing 
$(\cdot,\cdot)_\lambda:M_\lambda^+ \otimes M_\lambda^-\longrightarrow\C$, called the {\it Shapovalov pairing}, which we now decribe in more explicit terms. 
Using the PBW decomposition $U(\g^\C) \cong U(\mathfrak{n}_-) \otimes U(\mathfrak{h}^\C)\otimes U(\mathfrak{n}_+)$ one can define a projection 
$\phi: U(\g^\C) \rightarrow U(\mathfrak{h}^\C)$ by means of $\epsilon\otimes id\otimes\epsilon$: for any $x \in U(\g^\C)$, 
$$
x = \underbrace{\phi(x)}_{\in U(\mathfrak{h}^\C)} + \underbrace{x_-}_{\in n_-U(\g^\C)} + \underbrace{x_+}_{\in U(\g^\C) n_+}\,.
$$
Let $\nu_\lambda$, $\nu_{-\lambda}$ denote the generators of $M_\lambda^+$ and $M_\lambda^-$, respectively. For $x\in U(\mathfrak{n}_-)$ and $y \in U(\mathfrak{n}_+)$,
\begin{eqnarray*}
	x.\nu_\lambda \otimes y.\nu_{-\lambda} 
	& = &  S(y)x.\nu_\lambda \otimes \nu_{-\lambda}
	  =\big(\phi(S(y)x) + (S(y)x)_- + (S(y)x)_+\big).\nu_\lambda \otimes \nu_{-\lambda} \\
	& = & \big(\phi(S(y)x) + (S(y)x)_+\big).\nu_\lambda \otimes \nu_{-\lambda} \\
	& = & \phi(S(y)x).\nu_\lambda \otimes \nu_{-\lambda} + \nu_\lambda \otimes S((S(y)x)_+).\nu_{-\lambda} \\
	& = & \lambda(\phi(S(y)x)).\nu_\lambda \otimes \nu_{-\lambda}\,,
\end{eqnarray*}
where $S$ denotes the antipode on $U(\g^\C)$. It follows that $(x.\nu_\lambda , y.\nu_{-\lambda}) = \lambda\big(\phi(S(y)x)\big)$. 
By identifying $U(\mathfrak{n}_-) \cong M_\lambda^+$ and $U(\mathfrak{n}_+) \cong M_\lambda^-$ we get a family of pairings 
$(\cdot,\cdot)_\lambda:U(\mathfrak{n}_-)\otimes U(\mathfrak{n}_+) \rightarrow \C$ parametrized by characters on $\mathfrak{h}^\C$.
As is shown in \cite[Proposition 3.1]{AL} this pairing is non-degenerate for all $\lambda$ in the complement of a countable subset of $\h^\C$. 
Moreover, it respects the grading in the root lattice, implying the inverse pairing is given by an element 
$F_\lambda \in U(\mathfrak{n}_-)\widehat{\otimes} U(\mathfrak{n}_+)$ in the completed tensor product. Furthermore \cite{AL} shows that $F_\lambda$ is holomorphic 
at infinity, enabling us to define
$$
B := F_{\frac{\lambda}{t}} \in U(\mathfrak{n}_-)\widehat{\otimes} U(\mathfrak{n}_+)\llb t \rrb\,.
$$

\medskip

We now define a $\star$-product on $\mathcal{O}_\lambda$ as follows. Let $\pi: G \rightarrow \mathcal{O}_\lambda\,,\,g\mapsto g.\lambda$ denote the canonical projection 
and let $\g$ act on functions on $G$ by left-invariant vector fields: for $f \in C^\infty(G,\C)$ and $x \in \g$, 
$$
(x.f)(g) = \frac{d}{dt}_{\big|_{t=0}}\Big(f(g e^{t x})\Big)\,.
$$
Then, for $f,g \in C^\infty(\mathcal{O}_\lambda,\C)$, we define
$$
f \star g = m\Big(B.(\pi^*f \otimes \pi^*g)\Big)\in C^\infty(G,\C) \llb t \rrb\,,
$$
where $m$ denotes the usual multiplication of functions. We finally have the following: 
\begin{Thm}[Alekseev-Lachowska, \cite{AL}]
The product $\star$ indeed provides a deformation quantization of $\mathcal{O}_\lambda$. In particular, \\[-0.5cm]
\begin{itemize}
	\item $f \star g$ is $H$-invariant and thus defines a function on $\mathcal{O}_\lambda$. \\[-0.6cm]
	\item $\star$ is associative. \\[-0.6cm]
	\item $B = 1 + t \sum_{\alpha \in \Delta_+} e_{-\alpha} \otimes e_{\alpha} + O(t^2)$, and thus $[\cdot,\cdot]_\star = t \{\cdot,\cdot \} + O(t^2)$.
\end{itemize}
\end{Thm}

\subsection{Properties of the Alekseev-Lachowska $\star$-product}

The $\star$-product we have just defined has the following additional properties: 
\begin{Prop}\label{mainprop}
It is polarized, with the commutative subalgebra being 
$$
\mathcal{O}_t := \{ f \in C^\infty_{\mathcal O_\lambda} \llb t \rrb \, \big|\, e_\alpha.f = 0 \ \ \forall \alpha \in \Delta_-\}=\mathcal O_0\llb t\rrb\,.
$$
Moreover it has a strong quantum momentum map.
\end{Prop}
\begin{proof}
The first part is obvious since $B\in U(\mathfrak{n}_-)\widehat\otimes U(\mathfrak{n}_+)\llb t \rrb$. 
For the second part, recall that a $\star$-product on $\mathcal O_\lambda$ has a {\it strong quantum moment map} if, for any 
$h\in\mathfrak g^\C\subset C^\infty(\mathfrak g^*,\C)$, we have 
\begin{equation}\label{eq-sqmm}
[h_{|\mathcal O_\lambda},\cdot]_\star=t\{h_{|\mathcal O_\lambda},\cdot\}\,,
\end{equation}
and remember the following: 
\begin{Lem}[\cite{AL} Propostion 4.1]\label{3:uniq}
Let $\mathcal{V}$ be a $U(\g^\C)$-module and $w \in \mathcal{V}$. If $\lambda\in\mathfrak g^*$ is such that the Shapovalov pairing is non-degenerate, then there is at 
most one $U(\mathfrak{n}_+)$-invariant element $z\in M_\lambda^+\widehat{\otimes}\mathcal V$ of the form 
$z \in \nu_{\lambda} \otimes w + U(\mathfrak{n}_-)_{> 0}.\nu_{\lambda} \widehat{\otimes} \mathcal{V}$.
\end{Lem}
\begin{proof}[Proof of Lemma \ref{3:uniq}]
Let $(y_s)_s$ be a basis of $U(\mathfrak{n}_+)_{>0}$, which we choose to be homogeneous (w.r.t.~the root lattice grading). 
Since $(\cdot,\cdot)_\lambda$ is homogeneous and non-degenerate, then we can find a basis $(x_s)_s$ of $U(\mathfrak{n}_-)_{>0}$ such that 
$\lambda\big(\phi(y_lx_s)\big) = \delta_{ls}$. Now let $z$ be of the form
$$
z = \nu_{\lambda} \otimes w + \sum_s x_s.\nu_{\lambda} \otimes w_s\,.
$$
By $U(\mathfrak{n}_+)$-invariance we have, for any $l$, 
$$
0 = y_l.z=\nu_{\lambda} \otimes y_l.w + \sum_s x_s.\nu_{\lambda} \otimes y_l.w_s + \sum_s (y_l x_s).\nu_{\lambda} \otimes w_s.
$$
The homogeneous part of degree $0$ in $(y_l x_s).v_\lambda$ is $\phi(y_l x_s).v_\lambda=\lambda\big(\phi(y_l x_s)\big)v_\lambda$. Hence we have 
$$
- y_l.w = \sum_s \lambda\big(\phi(y_lx_s)\big) w_s = w_l\,.
$$
This proves the Lemma. 
\end{proof}
Let now $f_h:=h_{|\mathcal O_\lambda}$, pull it back to a function on $G$ and consider the action of $x \in \g$:
\begin{eqnarray*}
(x.\pi^*f_h) (g) & = & \frac{d}{dt}_{\big|_{t=0}}\Big((g e^{t x}).\lambda\Big)(h) = \frac{d}{dt}_{\big|_{t=0}}\Big(\lambda\big((g e^{t x})^{-1}.h\big)\Big) \\
		 & = & \lambda(-[x,g^{-1}.h]) = (x.\lambda)(g^{-1}.h)\,.
\end{eqnarray*}
It follows that the evaluation $ev_h$ at $g^{-1}.h$ defines a morphism of $\g$-modules from the $\g$-submodule of $\g^*$ generated by $\lambda$ to the $\g$-submodule 
of $C^\infty(\mathcal O_\lambda,\C)$ generated by $f_h$. 
Hence 
$$
f \star f_h=m\circ(id\otimes ev_h)\big(B.(f\otimes \lambda)\big)\,.
$$
\begin{Lem}\label{comp}
$f\star f_h=m(f,f_h)-t\sum_\alpha m(e_{-\alpha}.f,e_{\alpha}.f_h)$. 
\end{Lem}
\begin{proof}[Proof of Lemma \ref{comp}]
It is sufficient to show that 
$$
B.(\nu_{\frac{\lambda}{t}}\otimes\lambda)=\nu_{\frac{\lambda}{t}}\otimes\lambda-t\sum_\alpha e_{-\alpha}.\nu_{\frac{\lambda}{t}}\otimes e_{\alpha}.\lambda
\in M^+_{\frac{\lambda}t}\otimes\g^*\,.
$$
The l.h.s.~is the image of $F_{\frac{\lambda}t}$ under the following morphism of $U(\mathfrak{n}_+)$-modules: 
$$
M_{\frac{\lambda}t}^+ \widehat{\otimes} M_{\frac{\lambda}t}^- \longrightarrow M_{\frac{\lambda}t}^+ \widehat{\otimes} \g^*\,,\,
y.\nu_{\frac{\lambda}t}\otimes x.\nu_{-\frac{\lambda}t}\longmapsto y.\nu_{\frac{\lambda}t}\otimes x.\lambda\,.
$$
As such it is clearly $U(\mathfrak{n}_+)$-invariant. 
Hence by Lemma \ref{3:uniq} it is enough to show that the r.h.s., which we denote $\psi$, is $\mathfrak n_+$-invariant. Indeed, for any $\beta$, 
\begin{eqnarray*}
	e_{\beta}.\psi & = & \nu_\frac{\lambda}{t}\otimes e_{\beta}.\lambda - t \sum_\alpha 
	\Big(e_\beta e_{-\alpha}.\nu_\frac{\lambda}{t} \otimes e_\alpha.\lambda+e_{-\alpha}.\nu_\frac{\lambda}{t}\otimes e_\beta e_\alpha.\lambda\Big)\\
	& = & \nu_\frac{\lambda}{t}\otimes
	      \Big(\underbrace{e_\beta-t\sum_{\alpha,[e_\beta,e_{-\alpha}] \in \mathfrak{h}}\frac{\lambda}{t}([e_\beta,e_{-\alpha}]) e_\alpha}_{=0}\Big).\lambda\\
	&& \, -t\sum_\alpha\Big(e_{-\alpha}.\nu_\frac{\lambda}{t}\otimes e_\beta e_\alpha
	      +\sum_{\gamma}c_{\beta,-\alpha}^\gamma e_\gamma.\nu_{\frac{\lambda}t}\otimes e_\alpha\Big).\lambda \\
	& = & \, -t\sum_{\gamma}e_\gamma.\nu_{\frac{\lambda}t}\otimes\Big(e_\beta e_{-\gamma}+c_{\beta,\gamma-\beta}^{\gamma}e_{\beta-\gamma}\Big).\lambda\,.
\end{eqnarray*}
The last line vanishes because 
$\Big(e_\beta e_{-\gamma}+c_{\beta,\gamma-\beta}^{\gamma}e_{\beta-\gamma}\Big).\lambda
=\Big(e_\beta e_{-\gamma}+\lambda([e_{-\gamma},[e_\beta,e_{\gamma-\beta}])e_{\beta-\gamma}\Big).\lambda$, 
which is zero. 
\end{proof}
\noindent{\it End of the proof of Proposition \ref{mainprop}.} By Lemma \ref{comp} and its obvious analogue for $f_h$ sitting in the first argument we get
$f_h \star f - f \star f_h = t \sum_\alpha (e_{-\alpha} \wedge e_\alpha).(f_h \otimes f)$, proving the condition for a strong quantum momentum map.
\end{proof}
\begin{Rem}
Lemma \ref{comp} can be interpreted in analogy with geometric quantization as follows. Let us pretend that $\frac{\lambda}{t}$ is a dominant integral weight, 
so that there exists a line bundle $\mathcal{L}$ associated with the principal bundle $G \rightarrow G/H$. The splitting of $\h^\C \rightarrow \g^\C$ induces 
a connection on $\mathcal{L}$, whose curvature is given by $\frac{-\omega}{t}$. In the compact case, the choice of $\mathfrak n_{-}$ gives furthermore a holomorphic 
structure on both the coadjoint orbit and the line bundle. By the Borel-Weil Theorem the $\g^\C$-module of holomorphic sections is irreducible of highest weight equal 
to the first Chern class of $\mathcal L$.
Moreover, the formula in Lemma \ref{comp} can be interpreted as the prequantization of a function $f$. Namely, 
$f\mapsto \widehat{f}:= f + t \nabla_{X_f}$ ($X_f:=\{f,\cdot\}$) is such that $[\widehat{f},\widehat{g}] = t \widehat{\{f,g\}}$. 
In this context Lemma \ref{comp} claims that the Hamiltonian functions given by the momentum map, whose Hamiltonian flow leaves the polarization invariant, 
are quantizable in the sense of geometric quantization.
\end{Rem}

Recall that $B\in U(\mathfrak n_-)\widehat{\otimes}U(\mathfrak n_+)\llb t \rrb$. Therefore $\star$ obviously defines a deformation quantization with 
separation of variables, where 
$$
\widetilde{\mathcal{O}}_t := \{ f \in C^\infty_{\mathcal O_\lambda} \llb t \rrb \, \big|\, e_\alpha.f = 0 \ \ \forall \alpha \in \Delta_+\}
=\widetilde{\mathcal{O}}_0\llb t \rrb\,.
$$
Moreover, we have the following uniqueness statement: 
\begin{Prop}
There is a unique $G$-invariant $\star$-product on $\mathcal{O}_\lambda$ with separation of variables for a given $G$-invariant polarization that furthermore 
has a strong quantum momentum map. 
\end{Prop}
\begin{proof}
First of all, since the polarization is $G$-invariant, it is also $\h$-invariant, so is splits into one-dimensional root spaces. Thus, after a suitable choice of positivity, the above construction applies.
For uniqueness, since
$$
\Diff ^2_G (G/H) = ((U\g/ U\g \, \mathfrak{h} ) ^{\otimes 2}) ^H
$$
and $\star$ admits separation of variables, we can write
$$
f \star g = B(f \otimes g) 
\text{  for } B \in Un_- \widehat{\otimes} Un_+ \llb t \rrb\,.
$$
From this form it follows that the $\star$-commutator determines the product itself. In particular, left multiplication by a Hamiltonian momentum function is determined and concretely given by Lemma \ref{comp}. 
Now uniqueness follows from the fact that Hamiltonian momentum functions together with $1$ generate a subalgebra in $(C^\infty \llb t \rrb, \star)$ that is dense in any formal neighborhood.
\end{proof}

\subsection{Characteristic class of the Alekseev-Lachowska $\star$-product}

In this Subsection we compute the characteristic class of a $G$-invariant $\star$-product with separation of variable 
and strong quantum momentum map, with the help of Theorem \ref{BDthm}. 
\begin{Lem}
For every $h\in\g^\C$, $h_{|\cO}\in\mathcal{F}(\mathcal O_t)$. 
\end{Lem}
\begin{proof}
This follows from a simple computation: 
$$
\frac1t[h_{|\cO},\mathcal O_t]_\star=\{h_{|\cO},\mathcal O_t\}=h.\mathcal O_t\subset\mathcal O_t\,,
$$
where we made use of the strong quantum momentum map property \eqref{eq-sqmm}. 
\end{proof}
\begin{Prop}
The following morphism of $\mathcal{O}_t$-modules is onto: 
\begin{equation}\label{eq-map2}
\mathcal{O}_t\otimes (\C1 \oplus \g) \longrightarrow \mathcal{F}(\mathcal{O}_t)\,,\,f \otimes h \longmapsto f\star h_{|\cO} = fh_{|\cO}\,.
\end{equation}
Furthermore, there is a local commuting basis of $\mathcal{F}(\mathcal{O}_t)$ in $C^\infty_{\mathcal O_\lambda}$ (i.e.~independent of $t$).
Moreover, the Lie bracket on $\mathcal{F}(\mathcal{O}_t)$ coincides with the Poisson bracket.
\end{Prop}
\begin{proof}
Let us compose the morphism \eqref{eq-map2} with the local isomorphism \eqref{eq-map} of Proposition \ref{prop3.16}: we get 
$$
\g \longrightarrow \bigoplus_{i=1}^n (\mathcal{O}_t)_{|U} \,,\, x \longmapsto \Big(\frac{1}{t}[x_{|\cO},a_1]_\star,\cdots, \frac{1}{t}[x_{|\cO},a_n]_\star\Big).
$$
The strong quantum momentum map property ensures that 
$$
\frac{1}{t}[x_{|\cO},a_i]_\star = \{x_{|\cO},a_i\} = da_i(x).
$$
Since the $da_i$'s are independent and the Hamiltonian functions span the entire tangent space, this map becomes onto after tensoring with $C^\infty_U\llb t\rrb$. 
We thus have a morphism of free finitely generated $(\mathcal{O}_t)_{|U}$ modules that is onto after tensoring with $C^\infty_U \llb t \rrb$: it follows 
that the original map is onto as well. 
In order to get the statement about the Lie bracket we compute for $f,g \in \mathcal{O}_t$ and $x,y \in \g$
\begin{eqnarray*}
[fx_{|\cO},gy_{|\cO}]_\star & = & [f\star x_{|\cO},g\star y_{|\cO}]_\star \\
& = & [f,g]_\star\star x_{|\cO}\star y_{|\cO} + f\star[x_{|\cO},g]_\star\star y_{|\cO} + \\ 
&   & + g\star[f,y_{|\cO}]_\star\star x_{|\cO} + g\star f\star[x_{|\cO},y_{|\cO}]_\star \\
&=& t\big(f\{x_{|\cO},g\}y_{|\cO} + g\{f,y_{|\cO}\}x_{|\cO} + gf\{x_{|\cO},y_{|\cO}\} \big) \\
&=& t \{fx_{|\cO},gy_{|\cO}\}\,.
\end{eqnarray*}
Finally, by separation of variables, the $\mathcal{O}_t$-module structure on $\mathcal{F}(\mathcal{O}_t)$ is trivial, 
and since the $a_i$'s are independent of $t$, we can just take the $t=0$ part of the $f_i$'s in Proposition \ref{prop3.16}.
\end{proof}
\begin{Cor}\label{cor-3.8}
The class $cl(\mathcal{A}_t,\mathcal{O}_t)\in H^2(\mathcal O_\lambda,\C\llb t \rrb)$ does not depend on $t$.
\end{Cor}
\begin{proof}
The above Proposition states that there exists a local splitting which is independent of $t$. 
The induced characteristic class is then independent of $t$.
\end{proof}
Therefore we have the following: 
\begin{Thm}\label{thm-3.9}
The characteristic class $\theta$ of $\mathcal A_t$ is 
$$
\theta = [\omega] - \frac{t}{2}c_1(P_0)\,.
$$
\end{Thm}
\begin{proof}
It follows from Theorem \ref{BDthm} and Corollary \ref{cor-3.8}.
\end{proof}
\begin{Rem}
It should be noted, that neither separation of variables, nor the strong quantum momentum map property puts any restriction on the characteristic class of a $G$-invariant $\star$-product by itself. Namely, \cite{BD} construct $\star$-products with separation of variables of any characteristic class. And the following Lemma shows that for $G$ compact, we can always achieve the strong quantum momentum property.
\end{Rem}
\begin{Lem}
Let $G$ be compact, acting on $(M,\omega)$, and let $\star$ be a $G$-invariant $\star$-product. Then $\star$ is equivalent to a $\star$-product satisfying the strong quantum momentum map property.
\end{Lem}
\begin{proof}
The same proof as in the classical case shows that, $\g$ being semi-simple, there is a unique quantum momentum map. Namely, for every $\xi \in \g$ there is a function $\mu_\xi \in C^\infty(M)\llb t \rrb$ such that $[\mu_\xi, \mu_\nu]_\star = t \mu_{[\xi,\nu]}$. Let $\mu^0$ denote the $0$-th order part, that is the classical momentum map. Now we are seeking a differential operator $F$ that sends $\mu^0_\xi$ to $\mu_\xi$. To construct such a differential operator, we only need to be able to distinguish the different Hamiltonians by their value in a formal neighborhood. This is certainly possible, e.g. by exhibiting the Hamiltonian vector fields as Killing fields of a $G$-invariant metric, which are uniquely determined by their $1$-jets. After averaging by $G$, we can furthermore assume that $F$ is $G$-equivariant. Now, it is easily seen that the $\star$-product $F^{-1}(F(\cdot)\star F(\cdot))$ is $G$-invariant with a strong quantum momentum map.
\end{proof}


\section{Quantum dimension}

We borrow the notation from the previous Section and assume that $G$ is compact. In particular the integral in the r.h.s.~of \eqref{FNT} makes perfect sense. 
Moreover, the polarization in this case is purely complex and thus induces a K\"ahler structure on $\mathcal O_\lambda$. By possibly interchanging $\Delta_-$ and $\Delta_+$ we can assume that $P_0$ is the anti-holomorphic tangent bundle and thus left multiplication with holomorphic function is trivial.

\subsection{Reminder about the cohomology of a coadjoint orbit}

In this Subsection we provide a description of the cohomology of $\mathcal O_\lambda$, in particular $H^2(\cO,\R)$, in Lie theoretic terms. 
We consider the subcompex of $\g$-invariant differential forms, which is isomorphic to the relative Chevalley-Eilenberg complex $C^\bullet(\g,\h)$
$$
C^k(\g,\h) := \Hom_{\h\textrm{-mod}}\Big(\Lambda^k(\g/\h),\R\Big)\,.
$$
We denote the corresponding cohomology by $H^\bullet(\g,\h)$; for every $k$ there is a linear map 
\begin{equation}\label{eq1}
H^k(\g,\h) \longrightarrow H^k(\cO,\R)
\end{equation}
which happens to be an isomorphism as $G$ is compact (by averaging). Moreover, there is a map 
$$
\big(\g^*\big)^\h \longrightarrow C^2(\g,\h)\,,\, \alpha \longmapsto -\alpha([\cdot,\cdot])\,.
$$
It induces an isomorphism 
$$
\frac{\big(\g^*\big)^\h}{(\h^\perp)^\h}\longrightarrow H^2(\g,\h)
$$
which happens to be an isomorphism (one can prove it by direct computation or by considering the five-term exact sequence derived from the Hochschild-Serre 
spectral sequence for relative cohomology). As we have an $\h$-equivariant projection $p:\g\to\h$, the above gives an isomorphism 
\begin{equation}\label{eq2}
(\h^*)^\h=\Big(\frac{\h}{[\h,\h]}\Big)^* \longrightarrow  H^2(\g,\h)\,,\,\alpha \longmapsto -\alpha\big(p([\cdot,\cdot])\big)\,.
\end{equation}
It then follows from the definition of the symplectic form $\omega$ (see Example \ref{exa-ca}) that $[\omega]$ is the image of $-\lambda$ by the composition of \eqref{eq1} and \eqref{eq2} 
(here we abuse a bit the notation as we use the same symbol for a character and its restriction to $\h$). 

\medskip

Let $V$ be a complex $H$-module and consider the associated bundle $G\times_H V$ on $\mathcal O_\lambda=G/H$. 
\begin{Prop}
The first Chern class $c_1(G\times_H V)$ is the image of (the real part of) the character $i \, \chi_V$ of $V$ (recall $\chi_V(h):=Tr_V(h.\cdot)$) by the composition of \eqref{eq1} and \eqref{eq2}.
\end{Prop} 
\begin{proof}
Observe that $p$ represents a connection 1-form, which is principal since $p$ is $H$-invariant. Its curvature is 
$$
\Omega(\cdot,\cdot) = -p([\cdot,\cdot]) + [p(\cdot),p(\cdot)] \in \Hom_{\h\textrm{-mod}}\Big(\Lambda^2(\g/\h),\h\Big)\cong\Omega^2\big(\cO,G\times_H\h\big)^G\,.
$$
To get the curvature for the associated bundle $G\times_HV$ one just composes with the character $\chi_V$ and get 
$$
\Omega_V(\cdot,\cdot)=-\chi_V\big(p([\cdot,\cdot])\big)\,.
$$
This proves that $c_1(G\times_H V)$ is the image of $i \, \chi_V$ by the composition of \eqref{eq1} and \eqref{eq2}, or rather their complexifications. 
We then conclude by recalling that the first Chern class is real-valued. 
\end{proof}
Let us apply the above proposition to $V=\mathfrak{n}_-$, so that $G\times_HV=P_0$. We get that 
$c_1(P_0)$ is the image of $-i \, \sum_{\alpha \in \Delta^+} \alpha$.
We can therefore conclude, thanks to Theorem \ref{thm-3.9}, that the characteristic class of the Alekseev-Lachowska $\star$-product corresponds to 
$-\lambda + t \, i \, \rho$ under the composition of \eqref{eq1} and \eqref{eq2}. 

\subsection{Weyl dimension formula}

In this Subsection we compute $Tr(1)$ using Theorem \ref{thm-3.9} and the trace formula \eqref{FNT}. 
\begin{Thm}
\label{trformula}
$$Tr(1)= \frac{\prod_{\alpha \in \Delta_+} \langle \alpha, \frac{i}{t}\lambda + \rho \rangle}{\prod_{\alpha \in \Delta_+} \langle \alpha, \rho \rangle} .$$
\end{Thm}
\begin{proof}
Using the formula for the characteristic class in the form $\theta = \omega - \frac{t}{2}c_1(P_0)$ we see that
\[
Tr(1) = \frac{1}{(2\pi)^n} \int_{\cO} e^{\frac{1}{t}\omega - \frac{1}{2}c_1(P_0)} \ \hat{A}(TM,\omega) = \frac{1}{(2\pi)^n} \int_{\cO} e^{\frac{1}{t}\omega} \ \widetilde{Td}(M,\omega).
\]
Where $\widetilde{Td}(M,\omega)$ is the unnormalized Todd genus, differing from the standard Todd genus by $(2\pi)^{-k}$ in its $2k$-th degree term.
\begin{Rem}
We follow here the convention of \cite{BD,NT} in the definition of Chern classes (classes take values in $(2\pi)^k \Z$). 
Hence the Todd genus appearing here has a different normalization than the standard Todd genus appearing in the Riemann-Roch formula.
\end{Rem}
Making up the the different normalization, we get
\[
Tr(1) = \int_{\cO} e^{\frac{1}{2\pi t}\omega} \ {Td}(M,\omega),
\]
which is the Riemann-Roch formula for a line bundle with curvature equal to $\frac{1}{i t}\omega$. In analogy with geometric quantization, if we formally assume that 
$\frac{-\lambda}{i t}$ is a dominant integral character, the above formula gives the Euler characteristic of the associated holomorphic line bundle, and by the Borel-Weil 
theorem, this is the dimension of the irreducible highest weight module of weight $\frac{-\lambda}{i t}$. Now the result follows from the Weyl dimension formula, by noting that the l.h.s is polynomial in $\lambda$ and that the $\lambda$'s for which $\frac{i}{t}\lambda$ is dominant integral are Zariski dense.
\end{proof}
\begin{Rem}
One can also prove Theorem \ref{trformula} by using localization (see e.g.~\cite[Section 7.5]{BGV}), effectively getting the same calculation as for the 
symplectic volume of a coadjoint orbit, where one uses at some point that the contribution of the $\hat{A}$-genus is invariant under the Weyl group.
\end{Rem}

\newcommand{\bysame}{\leavevmode\hbox to3em{\hrulefill}\,}

\end{document}